\documentclass[12pt]{article}
\usepackage{amsmath,amssymb,amsthm, amsfonts}
\usepackage{etaremune, enumerate, float}
\usepackage{hyperref}
\usepackage{graphicx}
\usepackage{url}
\usepackage[mathlines]{lineno}
\usepackage{dsfont} 
\usepackage{tikz, graphicx, subcaption, caption}
\usepackage[algo2e,ruled,vlined]{algorithm2e}
\usepackage{color}
\definecolor{red}{rgb}{1,0,0}

\definecolor{blue}{rgb}{0,0,1}

\definecolor{green}{rgb}{0,.6,0}

\numberwithin{figure}{section}   

\newtheorem{thm}{Theorem}

\theoremstyle{definition}

\theoremstyle{definition}

\theoremstyle{definition}

\newcommand{\bit}{\begin{itemize}}
\newcommand{\eit}{\end{itemize}}
\newcommand{\ben}{\begin{enumerate}}
\newcommand{\een}{\end{enumerate}}
\newcommand{\beq}{\begin{equation}}
\newcommand{\eeq}{\end{equation}}
\newcommand{\bea}{\begin{eqnarray*}} 
\newcommand{\eea}{\end{eqnarray*}}
\newcommand{\bpf}{\begin{proof}}
\newcommand{\epf}{\end{proof}\ms}
\newcommand{\bmt}{\begin{bmatrix}}
\newcommand{\emt}{\end{bmatrix}}
\newcommand{\ms}{\medskip}

\title{A note on long rainbow arithmetic progressions}
\author{Jesse Geneson\\
\small\tt ISU\\
\small\tt geneson@gmail.com
}

\begin{document}

\maketitle

\begin{abstract}
Jungi\'{c} et al (2003) defined $T_{k}$ as the minimal number $t \in \mathbb{N}$ such that there is a rainbow arithmetic progression of length $k$ in every equinumerous $t$-coloring of $[t n]$ for every $n \in \mathbb{N}$. They proved that for every $k \geq 3$, $\lfloor \frac{k^2}{4} \rfloor <  T_{k} \leq \frac{k(k-1)^2}{2}$ and conjectured that $T_{k} = \Theta(k^2)$. We prove for all $\epsilon > 0$ that $T_{k} = O(k^{5/2+\epsilon})$ using the K\H{o}v\'{a}ri-S\'{o}s-Tur\'{a}n theorem and Wigert's bound on the divisor function.
\end{abstract}

A coloring is called \emph{equinumerous} if each color is used an equal number of times. We call an arithmetic progression \emph{rainbow} if it does not contain any terms of the same color. We use the notation $AP(k)$ for arithmetic progression of length $k$, so $T_{k}$ is the minimal number $t \in \mathbb{N}$ such that there is a rainbow $AP(k)$ in every equinumerous $t$-coloring of $[t n]$ for every $n \in \mathbb{N}$.

Besides the bounds $\lfloor \frac{k^2}{4} \rfloor <  T_{k} \leq \frac{k(k-1)^2}{2}$ for every $k \geq 3$ from \cite{jungic}, it is also known that $T_{k} > k$ for every $k > 5$ \cite{AF}, $T_{3} = 3$ \cite{AF, JR}, and $T_{4} > 4$ \cite{CJR}. A variant of $T_{k}$ in which the coloring is not required to be equinumerous was investigated in \cite{butler}. In this note, we prove for all $\epsilon > 0$ that $T_{k} = O(k^{5/2+\epsilon})$. 

Our proof represents the occurrences of a single color in $[t n]$ with a family of 0-1 matrices, and we bound $T_{k}$ using the K\H{o}v\'{a}ri-S\'{o}s-Tur\'{a}n theorem. We say that a 0-1 matrix $A$ \emph{avoids} a 0-1 matrix $P$ if there is no submatrix in $A$ that is equal to $P$ or that can be transformed into $P$ by changing some ones to zeroes. Define $ex(m, n, P)$ as the maximum number of ones in an $m \times n$ 0-1 matrix that avoids $P$. We let $R_{s, t}$ denote the $s \times t$ matrix of all ones. 

\begin{thm}\label{KSTt} \cite{KST}
For all $m, n, s, t \geq 2$, $ex(m, n, R_{s, t}) \leq (s-1)^{1/t} (n-t+1) m^{1-1/t}+(t-1)m$.
\end{thm}

Let $\tau(n)$ be the number of positive integer divisors of $n$. We use Wigert's upper bound on $\tau(n)$ \cite{wigert} for our bound on $T_{k}$.

\begin{thm}\label{wigertT} \cite{wigert}
$\tau(n) = n^{O(\frac{1}{\ln{\ln{n}}})}$
\end{thm}

Before we prove the new bound, we start with the proof of the $O(k^3)$ upper bound from \cite{jungic} since the beginning of our proof is the same. There was a small typo in the proof in \cite{jungic} for the value of the number of $AP(k)$s in $[m]$, which we correct in the version below.

\begin{thm}\cite{jungic}
$T_{k} = O(k^3)$
\end{thm}
\begin{proof}
Let $m = a(k-1)+b$ with $k \geq 3$, $a \geq 1$, and $0 \leq b < k-1$. Since there is a bijective correspondence between the set of all $AP(k)$s and the set of all $2$-element sets $\left\{x, y \right\} \subset [m]$, $x < y$, with $x \equiv y \pmod{k-1}$, the number of all $AP(k)$s in $[m]$ is $b \binom{a+1}{2} + (k-b-1) \binom{a}{2}$, implying that the number of $AP(k)$s in $[t n]$ is greater than $\frac{t n(t n - 3(k-1))}{2(k-1)}$. 

For any equinumerous $t$-coloring of $[t n]$, in each color there are $\binom{n}{2}$ pairs of numbers. Each pair can be terms of at most $\binom{k}{2}$ different $AP(k)$s. Thus for any equinumerous $t$-coloring of $[t n]$, there are at most $t \binom{k}{2} \binom{n}{2}$ $AP(k)$s that are not rainbow. Thus $T_{k}$ is bounded by the smallest $t$ satisfying $\frac{t n(t n - 3(k-1))}{2(k-1)} \geq t \binom{k}{2} \binom{n}{2}$ for all $n$, giving the bound.
\end{proof}

\begin{thm}
$T_{k} = O(k^{5/2+\epsilon})$ for all $\epsilon > 0$
\end{thm}

\begin{proof}
Let $t$ be minimal so that there is a rainbow $AP(k)$ in every equinumerous $t$-coloring of $[t n]$ for all $n \in \mathbb{N}$. From the bounds proved in \cite{jungic}, we may assume that $t < k^3$.

The overall method of the proof is the same as Jungi\'{c} et al \cite{jungic}, until the point where we find an upper bound on the number of $AP(k)$s in $[t n]$ that are not rainbow. For this bound, we split into cases depending on whether $n \leq k$.

If $n \leq k$, then $t n < k^4$. In each color, there are $\binom{n}{2}$ pairs of numbers. The difference between the numbers in each pair is less than $k^4$, so that difference has at most $k^{o(1)}$ divisors by Theorem \ref{wigertT}. Thus each pair can be terms of at most $k^{1+o(1)}$ different $AP(k)$s. Therefore for any $t$-regular coloring of $[t n]$, there are at most $t \binom{n}{2} k^{1+o(1)}$ $AP(k)$s that are not rainbow. The smallest $t$ satisfying $\frac{t n(t n - 3(k-1))}{2(k-1)} \geq t \binom{n}{2} k^{1+o(1)}$ for all $n \leq k$ is $O(k^{2+\epsilon})$ for all $\epsilon > 0$.

Now if $n > k$, fix an arbitrary color $c_i$ and let $a_1 < \dots < a_n$ be the $n$ numbers in $[t n]$ colored with $c_i$. We call the gap between $a_{j}$ and $a_{j+1}$ \emph{wide} if $a_{j+1}-a_{j} > k^{8}$. By the pigeonhole principle, there are at most $\frac{t n}{k^8} < \frac{n}{k^5}$ wide gaps between the consecutive numbers colored with $c_i$. 

Partition the numbers colored with $c_i$ into \emph{blocks} of $k$ consecutive numbers, where only the last block may have fewer than $k$ numbers. We say a block is wide if it contains both of the consecutive numbers in a wide gap. Since there are at most $\frac{n}{k^5}$ wide gaps between the consecutive numbers colored with $c_i$, there are at most $\frac{n}{k^5}$ wide blocks. In all other blocks, all of the gaps are not wide. 

Fix $1 \leq j \leq n-1$. The number of $a_{r}$ that lie in a wide block is at most $\frac{n}{k^4}$ since there are at most $\frac{n}{k^5}$ wide blocks and each block contains at most $k$ numbers. Thus there are at most $\frac{n}{k^4} \binom{k}{2}$ $AP(k)$s containing both $a_{j}$ and $a_{r}$ for some $r$ such that $j < r$ and $a_{r}$ lies in a wide block.

Next we bound the number of $AP(k)$s containing both $a_{j}$ and $a_{r}$ for some $r$ such that $j < r$ and $a_{r}$ lies in a non-wide block. Define a family of 0-1 matrices $\left\{A^{s}\right\}_{1 \leq s \leq \lceil \frac{n}{k} \rceil}$ corresponding to the blocks, each with $k-1$ rows and at most $k$ columns, such that $A^{s}_{x, y} = 1$ if and only if $x$ divides $a_{(s-1)k+y}-a_{j}$. 

For any fixed $r$ with $j < r$, the number of $AP(k)$s containing both $a_{j}$ and $a_{r}$ is at most $k \tau(a_{r}-a_{j})$. Thus the number of $AP(k)$s containing both $a_{j}$ and $a_{r}$ for some $r$ such that $j < r$ and $a_{r}$ lies in the block $s$ is at most $k$ times the number of ones in $A^{s}$. 

By Theorem \ref{wigertT}, $A^{s}$ avoids $R_{q, 2}$, where $q = k^{o(1)}$ for non-wide $s$. By the K\H{o}v\'{a}ri-S\'{o}s-Tur\'{a}n theorem, $A^{s}$ has at most $k^{3/2+o(1)}+k$ ones for non-wide $s$. Thus the number of $AP(k)$s containing both $a_{j}$ and $a_{r}$ for some $r$ such that $j < r$ and $a_{r}$ lies in the non-wide block $s$ is at most $k^{5/2+o(1)}+k^2$. 

Since there are at most $\lceil \frac{n}{k} \rceil \leq \frac{2n}{k}$ non-wide blocks, there are at most $t n (\frac{2n}{k} (k^{5/2+o(1)}+k^2) + \frac{n}{k^4} \binom{k}{2})$ $AP(k)$s in $[t n]$ that are not rainbow. The smallest $t$ satisfying $\frac{t n(t n - 3(k-1))}{2(k-1)} \geq t n (\frac{2n}{k} (k^{5/2+o(1)}+k^2) + \frac{n}{k^4} \binom{k}{2})$ for all $n > k$ is $O(k^{5/2+\epsilon})$ for all $\epsilon > 0$, completing the proof.
\end{proof}

\end{document}